\documentclass[a4paper,12pt]{amsart}
\usepackage{amssymb}
\usepackage{ifthen}
\usepackage{cite}
 \usepackage[dvips]{graphicx}
\nonstopmode \numberwithin{equation}{section}
\usepackage{amssymb}
\usepackage{ifthen}
\usepackage{graphicx}
\usepackage{amsmath}
\usepackage[T1]{fontenc} %skandit
\usepackage[utf8]{inputenc}
\usepackage[usenames,dvipsnames]{color}
\usepackage{color}
\usepackage[english]{babel}
\usepackage{fancyhdr}
\usepackage{fancybox}
\usepackage{tikz}
\nonstopmode \numberwithin{equation}{section}
\theoremstyle{plain}

\newtheorem{conj}{Conjecture}

\theoremstyle{definition}

\newtheorem{example}{Example}[section]
\newtheorem{thm}{Theorem}[section]
\newtheorem{prob}{Problem}[section]
\newtheorem{cor}{Corollary}[section]

\newtheorem{prop}{Proposition}[section]
\newtheorem{rem}{Remark}[section]
\newtheorem{lem}{Lemma}[section]

\newcounter{minutes}
\divide\time by 60
\newcounter{hours}
\multiply\time by 60
\addtocounter{minutes}{-\time}
\newcounter {own}
\def\theown {\thesection       .\arabic{own}}

\newenvironment{pf}[1][]{%
 \vskip 3mm
 \noindent
 \ifthenelse{\equal{#1}{}}%
  {{\slshape Proof. }}%
  {{\slshape #1.} }%
 }%
{\qed\bigskip}

\newcounter{alphabet}

\def\be{\begin{equation}}
\def\ee{\end{equation}}

\newcommand{\bee}{\begin{enumerate}}
\newcommand{\eee}{\end{enumerate}}

\newcommand{\blem}{\begin{lem}}
\newcommand{\elem}{\end{lem}}
\newcommand{\bthm}{\begin{thm}}
\newcommand{\ethm}{\end{thm}}
\newcommand{\bcor}{\begin{cor}}
\newcommand{\ecor}{\end{cor}}
\newcommand{\beg}{\begin{examp}}
\newcommand{\eeg}{\end{examp}}
\newcommand{\begs}{\begin{examples}}
\newcommand{\eegs}{\end{examples}}

\newcommand{\bdefn}{\begin{defn}}
\newcommand{\edefn}{\end{defn}}

\newcommand{\bprob}{\begin{prob}}
\newcommand{\eprob}{\end{prob}}
\newcommand{\bei}{\begin{itemize}}
\newcommand{\eei}{\end{itemize}}

\newcommand{\bcon}{\begin{conj}}
\newcommand{\econ}{\end{conj}}
\newcommand{\bcons}{\begin{conjs}}
\newcommand{\econs}{\end{conjs}}
\newcommand{\bprop}{\begin{prop}}
\newcommand{\eprop}{\end{prop}}
\newcommand{\br}{\begin{rem}}
\newcommand{\er}{\end{rem}}
\newcommand{\brs}{\begin{rems}}
\newcommand{\ers}{\end{rems}}
\newcommand{\bo}{\begin{obser}}
\newcommand{\eo}{\end{obser}}
\newcommand{\bos}{\begin{obsers}}
\newcommand{\eos}{\end{obsers}}
\newcommand{\bpf}{\begin{pf}}
\newcommand{\epf}{\end{pf}}
\newcommand{\ba}{\begin{array}}
\newcommand{\ea}{\end{array}}
\newcommand{\beq}{\begin{eqnarray}}
\newcommand{\beqq}{\begin{eqnarray*}}
\newcommand{\eeq}{\end{eqnarray}}
\newcommand{\eeqq}{\end{eqnarray*}}

\begin{document}

\title{Sharp bounds for second Hankel determinant of logarithmic coefficients for certain classes of univalent functions}
\author{Sanju Mandal}
\address{Sanju Mandal, Department of Mathematics, Jadavpur University, Kolkata-700032, West Bengal,India.}
\email{sanjum.math.rs@jadavpuruniversity.in}

\author{Partha Pratim Roy}
\address{Partha Pratim Roy, Department of Mathematics, Jadavpur University, Kolkata-700032, West Bengal,India.}
\email{pproy.math.rs@jadavpuruniversity.in}

\author{Molla Basir Ahamed}
\address{Molla Basir Ahamed, Department of Mathematics, Jadavpur University, Kolkata-700032, West Bengal,India.}
\email{mbahamed.math@jadavpuruniversity.in}

\subjclass[{AMS} Subject Classification:]{Primary 30A10, 30H05, 30C35, Secondary 30C45}

\def\thefootnote{}
\footnotetext{ {\tiny File:~\jobname.tex,
printed: \number\year-\number\month-\number\day,
          \thehours.\ifnum\theminutes<10{0}\fi\theminutes }
} \makeatletter\def\thefootnote{\@arabic\c@footnote}\makeatother
\begin{abstract} 
The Hankel determinant $H_{2,2}(F_{f}/2)$ is defined as:
\begin{align*}
	H_{2,2}(F_{f}/2):= \begin{vmatrix}
		\gamma_2 & \gamma_3 \\
		\gamma_3 & \gamma_4
	\end{vmatrix},
\end{align*}
where $\gamma_2, \gamma_3,$ and $\gamma_4$ are the second, third, and fourth logarithmic coefficients of functions belonging to the class $\mathcal{S}$ of normalized univalent functions. In this article, we establish sharp inequalities $|H_{2,2}(F_{f}/2)|\leq (1272 + 113\sqrt{678})/32856$ and $|H_{2,2}(F_{f}/2)| \leq 13/1080$ for the logarithmic coefficients of starlike and convex functions with respect to symmetric points. Moreover, we provide examples that demonstrate the strict inequality holds. 
\end{abstract}
\maketitle
\pagestyle{myheadings}
\markboth{S. Mandal, P. P. Roy, and M. B. Ahamed}{Hankel determinant of logarithmic coefficients for certain classes of univalent functions}

\section{Introduction}
Suppose $\mathcal{H}$ be the class of functions $ f $ which are holomorphic in the open unit disk $\mathbb{D}=\{z\in\mathbb{C}: |z|<1\}$ of the form 
\begin{align}\label{eq-1.1}
	f(z)=\sum_{n=1}^{\infty}a_nz^n,\; \mbox{for}\; z\in\mathbb{D}.
\end{align}
Then $\mathcal{H}$ is a locally convex topological vector space endowed with the topology of uniform convergence over compact subsets of $\mathbb{D}$. Let $\mathcal{A}$ denote the class of functions $f\in\mathcal{H}$ such that $f(0)=0$ and $f^{\prime}(0)=1$. Let $\mathcal{S}$ denote the subclass of all functions in $\mathcal{A}$ which are univalent.

Let
\begin{align}\label{eq-1.2}
	F_{f}(z):=\log\dfrac{f(z)}{z}=2\sum_{n=1}^{\infty}\gamma_{n}(f)z^n, \;\; z\in\mathbb{D},\;\;\log 1:=0,
\end{align}
be a logarithmic function associated with $f\in\mathcal{S}$. The numbers $\gamma_{n}:=\gamma_{n}(f)$ are called the logarithmic coefficients of $f$. Although the logarithmic coefficients $\gamma_{n}$ play a critical role in the theory of univalent functions, it appears that there are only a limited number of exact upper bounds established for them. As is well known, the logarithmic coefficients play a crucial role in Milin’s conjecture (\cite{Milin-1977-ET}, see also\cite[p.155]{Duren-1983-NY}). We note that for the class $\mathcal{S}$ sharp estimates are known only for $\gamma_{1}$ and $\gamma_{2}$, namely
\begin{align*}
	|\gamma_{1}|\leq 1, \;\; |\gamma_{2}|\leq \dfrac{1}{2}+ \dfrac{1}{e} =0.635\ldots
\end{align*}
Estimating the modulus of logarithmic coefficients for $f\in\mathcal{S}$ and various sub-classes has been considered recently by several authors. We refer to the articles \cite{Ali-Allu-PAMS-2018,Ali-Allu-Thomas-CRMCS-2018,Cho-Kowalczyk-kwon-Lecko-Sim-RACSAM-2020,Girela-AASF-2000,Thomas-PAMS-2016} and references therein.\vspace{1.2mm}

A function $f(z)$ is said to be starlike in $\mathbb{C}$ if it maps the open unit disk $|z|<1$ conformally onto a region that is starlike with respect to the origin. The class of functions that are starlike with respect to symmetric points, which is denoted as $\mathcal{S}^*_S$, was introduced by Sakaguchi in $1959$ \cite{Sakaguchi-JMSJ-1959}. A function $f$ that belongs to $\mathcal{S}^*_S$ is characterized by the following conditions:
\begin{align*}
	\mbox{Re}\left(\dfrac{zf^{\prime}(z)}{f(z)-f(-z)}\right)> 0, \;\;z\in\mathbb{D}.
\end{align*}

A function $f(z)$ is said to be convex in $\mathbb{C}$ if it maps the open unit disk $|z|<1$ conformally onto a region that is convex. We consider another class, which is denoted by $\mathcal{K}_S$, that is, a function $f\in\mathcal{A}$ is said to be convex with respect to symmetric points if, and only if,
\begin{align*}
	\mbox{Re}\left(\dfrac{(zf^{\prime}(z))^{\prime}}{(f(z)-f(-z))^{\prime}}\right)> 0, \;\;z\in\mathbb{D}.
\end{align*}
The functions that are members of $\mathcal{S}^*_S$ are identified as close-to-convex, and consequently, they are univalent. The class of functions that are starlike with respect to symmetric points also includes the classes of convex functions and odd starlike functions with respect to the origin. This study is dedicated to providing the sharp bound for the second Hankel determinant, whose entries are the logarithmic coefficients. We commence by presenting the definitions of Hankel determinants in the case where $f\in \mathcal{A}$.\vspace{1.2mm}

The Hankel determinant $H_{q,n}(f)$ of Taylor's coefficients of functions $f\in\mathcal{A}$ represented by \eqref{eq-1.1} is defined for $q,n\in\mathbb{N}$ as follows:
\begin{align*}
	H_{q,n}(f):=\begin{vmatrix}
		a_{n} & a_{n+1} &\cdots& a_{n+q-1}\\ a_{n+1} & a_{n+2} &\cdots& a_{n+q} \\ \vdots & \vdots & \vdots & \vdots \\ a_{n+q-1} & a_{n+q} &\cdots& a_{n+2(q-1)}
	\end{vmatrix}.
\end{align*}
Kowalczyk and Lecko \cite{Kowalczyk-Lecko-BAMS-2022} recently proposed a Hankel determinant whose elements are the logarithmic coefficients of $f\in\mathcal{S}$, realizing the extensive use of these coefficients. This determinant is expressed as follows:
\begin{align*}
	H_{q,n}(F_{f}/2)=\begin{vmatrix}
		\gamma_{n} & \gamma_{n+1} &\cdots& \gamma_{n+q-1}\\ \gamma_{n+1} & \gamma_{n+2} &\cdots& \gamma_{n+q} \\ \vdots & \vdots & \vdots & \vdots \\ \gamma_{n+q-1} & \gamma_{n+q} &\cdots& \gamma_{n+2(q-1)}
	\end{vmatrix}.
\end{align*}

The study of Hankel determinants with logarithmic coefficients of starlike or many other functions has been done extensively, their sharp bounds have been established. Recently the second Hankel determinants with logarithmic coefficients have been examined for selected subclasses of starlike functions, convex functions, univalent function, strongly starlike and strongly convex functions(see \cite{Allu-Arora-Shaji-MJM-2023,Kowalczyk-Lecko-BAMS-2022,Kowalczyk-Lecko-BMMS-2022} and references therein). Recently, the sharp bounds of $H_{2,1}(F_f/2)$ are established in \cite{Allu-Arora-Shaji-MJM-2023} for the classes $\mathcal{S}^*_S$ and $\mathcal{K}_S$. However, for further details on other aspects of the Hankel determinant, we suggest referring to the article \cite{Raza-Riza-Thomas-BAMS-2023} and its cited sources. The significance of logarithmic coefficients makes the proposed problem worth considering and interesting. The results enlarge the scope of knowledge on logarithmic coefficients. \vspace{1.2mm} 

The main objective of this paper is to continue the previous research work by examining the exact bound of the Hankel determinant $H_{2,2}(F_{f}/2)= \gamma_{2}\gamma_{4} -\gamma^2_{3}$ for two class of functions: starlike and convex functions with respect to symmetric points. Differentiating \eqref{eq-1.2} and using \eqref{eq-1.1}, a simple computation shows that 
\begin{align}\label{Leq-1.1}
	\begin{cases}
		\gamma_{1}=\dfrac{1}{2}a_{2},\vspace{1.5mm}\\ \gamma_{2}=\dfrac{1}{2} \left(a_{3} -\dfrac{1}{2}a^2_{2}\right), \vspace{1.5mm}\\ \gamma_{3} =\dfrac{1}{2}\left(a_{4}- a_{2}a_{3} +\dfrac{1}{3}a^3_{2}\right), \vspace{1.5mm}\\ \gamma_{4}= \dfrac{1}{2} \left(a_{5} -a_{2}a_{4} +a^2_{2} a_{3} -\dfrac{1}{2}a^2_{3} -\dfrac{1}{4}a^4_{2}\right), \vspace{1.5mm}\\ \gamma_{5}= \dfrac{1}{2}\left(a_{6} -a_{2}a_{5} -a_{3}a_{4} +a_{2} a^2_{3} + a^2_{2} a_{4} -a^3_{2} a_{3} +\dfrac{1}{5}a^5_{2}\right).
	\end{cases}
\end{align}
It is known that for the Koebe function $f(z)=z/(1-z)^2$, the logarithmic coefficients are $\gamma_{n}=1/n$, for each positive integer $n$. Since the Koebe function appears as an extremal function in many problems of geometric theory of analytic functions, one could expect that $\gamma_{n}=1/n$ holds for functions in $\mathcal{S}$. But this is not true in general, even in order of magnitude. The problem of computing
the bound of the logarithmic coefficients is also considered in \cite{Ali-Allu-PAMS-2018,Ali-Allu-Thomas-CRMCS-2018,Cho-Kowalczyk-kwon-Lecko-Sim-RACSAM-2020,Thomas-PAMS-2016,Ponnusamy-Sugawa-BDSM-2021} for several subclasses of close-to-convex functions. In 2021, Zaprawa \cite{Zaprawa-BSMM-2021} obtained the sharp bounds of the initial logarithmic coefficients $\gamma_{n}$ for functions in the classes $\mathcal{S}^*_S$  and $\mathcal{K}_S$.\vspace{2mm}

However, the primary aim of this paper is to find the sharp bound of
\begin{align*}
 H_{2,2}(F_{f}/2)&=\gamma_{2}\gamma_{4} -\gamma^2_{3}\\&=\dfrac{1}{288} \left(a^6_2 -6a^4_2 a_3 -12a^3_2 a_4 +72a_2 a_3 a_4 +18a^2_2 a^2_3 -36a^2_2 a_5 \right.\\&\nonumber\left.\quad -36a^3_3 -72a^2_4 +72a_3 a_5\right)
\end{align*} 
when $f$ is a member of the classes $\mathcal{S}^*_S$ and $\mathcal{K}_S$, which respectively refer to starlike functions and convex functions with respect to symmetric points. In addition, we give examples of functions to illustrate these results.

\section{\bf Main results}
Suppose $\mathcal{B}_0$ be the class of Schwarz function \textit{i.e} analytic function $w:\mathbb{D}\rightarrow\mathbb{D}$ such that $w(0)=0$. A function $w\in\mathcal{B}_0$ can be written as a power series
\begin{align*}
	w(z)=\sum_{n=1}^{\infty}c_n z^n.
\end{align*}
For two functions $f$ and $g$ which are analytic in $\mathbb{D}$, we say that the function $f$ is subordinate to $g$ in $\mathbb{D}$ and written as
$f(z)\prec g(z)$ there exists a Schwarz function $w\in\mathcal{B}_0$ such that $f(z)=g(w(z))$, $z\in\mathbb{D}$. To prove our results, the following lemma for Schwarz functions will play a key role.
\begin{lem}\cite{Efraimidis-JMAA-2016}\label{lem-2.1}
Let $ w(z)=c_1 z +c_2 z^2 + c_3 z^3+\ldots $ be a Schwarz function. Then 	
\begin{align*}
	|c_1|\leq 1, \;\;|c_2|\leq 1 -|c_1|^2, \;\; |c_3|\leq 1 -|c_1|^2 -\dfrac{|c_2|^2}{1+|c_1|}\;\;\mbox{and}\;\; |c_4|\leq 1 -|c_1|^2 -|c_2|^2.
\end{align*}
\end{lem}
\subsection{\bf Second Hankel determinant of logarithmic coefficients of $f\in \mathcal{S}^*_S$:}
We obtain the following sharp bound for $ H_{2,2}(F_{f}/2) $ for functions in the class $ \mathcal{S}^{*}_S $.
\begin{thm}\label{th-2.1}
Let $ f\in \mathcal{S}^{*}_S $. Then
\begin{align*}
	|H_{2,2}(F_{f}/2)|\leq\dfrac{(1272 +113\sqrt{678})}{32856}.
\end{align*}
The inequality is sharp.	
\end{thm}
\begin{proof}
Let $ f\in \mathcal{S}^{*}_S $ be of the form $ f(z)=z +\sum_{n=2}^{\infty} a_n z^n $, $ z\in \mathbb{D} $. Then by the definition of subordination there exits a Schwarz function $ w(z)=\sum_{n=1}^{\infty} c_n z^n $ such that
\begin{align}\label{eq-2.1}
	\dfrac{2zf^{\prime}(z)}{f(z)-f(-z)} =\dfrac{1 +w(z)}{1-w(z)}.
\end{align}
By equating the coefficients on each side of \eqref{eq-2.1}, we get
\begin{align}\label{eq-2.2}
	\begin{cases}
		a_{2} =c_1 ,\\ a_{3}= c_{2} + c^2_1,\vspace{1.5mm}\\ a_{4}=\dfrac{1}{2} (c_{3} +3c_1 c_{2} +2c^3_1),\vspace{1.5mm}\\ a_{5}= \dfrac{1}{2}(c_{4} +2c^2_{2} +5c^2_{1} c_{2} +2c_{1}c_{3} +2c^4_{1}).
	\end{cases}
\end{align}
In view of \eqref{Leq-1.1} and \eqref{eq-2.2}, a simple computation shows that
\begin{align}\label{eq-2.3}
	\nonumber H_{2,2}(F_{f}/2)&=\gamma_{2}\gamma_{4} -\gamma^2_{3}\\ \nonumber &=\dfrac{1}{288} \left(a^6_2 -6a^4_2 a_3 -12a^3_2 a_4 +72a_2 a_3 a_4 +18a^2_2 a^2_3 -36a^2_2 a_5 \right.\\&\nonumber\left. \quad -36a^3_3 -72a^2_4 +72a_3 a_5\right) \\&=\dfrac{1}{288}\left(c^6_1 +30c^4_1 c_2 -6c^3_1 c_3 + 72c^2_1 c^2_2 +18c^2_1 c_4 +36c^3_2\right. \\&\nonumber\left.\quad -18c^2_3 +36c_2 c_4\right)
\end{align}
Using the Lemma \ref{lem-2.1} into \eqref{eq-2.3}, we obtain
\begin{align}\label{eq-2.4}
	\nonumber 288|H_{2,2}(F_{f}/2)|&\leq |c_1|^6 +30|c_1|^4|c_2| +6|c_1|^3 \left(1-|c_1|^2 -\dfrac{|c_2|^2}{1+|c_1|}\right) +72|c_1|^2|c_2|^2 \\&\nonumber\quad +18|c_1|^2(1-|c_1|^2-|c_2|^2) +36|c_3|^3 + 18\left(1-|c_1|^2 -\dfrac{|c_2|^2}{1+|c_1|}\right)^2 \\&\quad +36|c_2|(1-|c_1|^2-|c_2|^2).
\end{align}
Suppose that $x=|c_1|$ and $y=|c_2|$. Then it follows from \eqref{eq-2.4} that
\begin{align}\label{eq-2.5}
	288|H_{2,2}(F_{f}/2)|\leq M(x,y),
\end{align}
where 
\begin{align*}
	M(x,y)&:=x^6 +30x^4 y + 6x^3\left(1-x^2 -\dfrac{y^2}{1+x}\right) +72x^2y^2 + 18x^2(1-x^2-y^2) \\&\quad +36y^3 +18\left(1-x^2 -\dfrac{y^2}{1+x}\right)^2 +36y (1- x^2 -y^2).
\end{align*}
In view of Lemma \ref{lem-2.1}, the region of variability of a pair $ (x,y) $ coincides with the set 
\begin{align*}
	\Omega=\{(x,y):0\leq x\leq 1, 0\leq y\leq 1-x^2\}.
\end{align*}
The goal is to establish the maximum value of $M(x,y)$ in the region $\Omega$. Therefore, the critical point of $ M(x,y) $ satisfies the conditions
\begin{align*}
	\dfrac{\partial M}{\partial x}&= 6x^5 -30x^4 +18x^2 -36x +120x^3y +108xy^2 +\dfrac{18xy^2(4-x)}{(1+x)} \\&\quad + \dfrac{6y^2(6-6x^2 +x^3)}{(1+x)^2} -\dfrac{36y^4}{(1+x)^3} = 0
\end{align*}
and
\begin{align*}
	\dfrac{\partial M}{\partial y}= 30x^4 -36x^2 +108x^2 y +36 + \dfrac{12y(6x^2 -x^3 -6)}{(1+x)} +\dfrac{72y^3}{(1+x)^2}= 0.
\end{align*}
There are no solutions of $M(x,y)$ inside the interior of $\Omega$, hence it is not possible for the function to attain a maximum value within this region. Since $M(x,y)$ is a continuous function on a compact set $\Omega$, its maximum value must occur at some point on the boundary of $\Omega$. On the boundary of $ \Omega $, a simple computation shows that
\begin{align*}
	&M(x,0)= x^6 -6x^5 +6x^3 -18x^2 +18 \leq 18 \;\;\mbox{for}\; 0\leq x\leq 1,\\& M(0,y)= 18y^4 -36y^2 +36y +18 \leq 36 \;\;\mbox{for}\; 0\leq y\leq 1
\end{align*} 
and
\begin{align*}
     M(x,1-x^2)= 37x^6 -90x^4 +18x^2 +36 \leq \dfrac{12(1272 +113\sqrt{678})}{1369} \;\;\mbox{for}\; 0\leq x\leq 1.
\end{align*}
Therefore, we see that
\begin{align*}
	\max_{(x,y)\in \Omega} M(x,y) =\dfrac{12(1272 +113\sqrt{678})}{1369}
\end{align*}
and it fllows from \eqref{eq-2.5} that
\begin{align}\label{eq-2.6}
	|H_{2,2}(F_{f}/2)|\leq \dfrac{(1272 +113\sqrt{678})}{32856}.
\end{align}	

To prove the equality in \eqref{eq-2.6} sharp, for 
\begin{align*}
	A=8696056 &+741393\sqrt{678} \\&+ 81\sqrt{11507173440 + 1965308656 \sqrt{678} + 83777409 (\sqrt{678})^2},
\end{align*}
we consider the following function
\begin{align*}
	f_{1}(z)&=\dfrac{z}{1 -\left(\sqrt{\frac{8}{27} +\frac{32\times 74^{2/3}}{27A^{1/3}} +\frac{2^{1/3}A^{1/3}}{27\times 37^{2/3}}}\right)z}\\&=z + \left(\sqrt{\frac{8}{27} +\frac{32\times 74^{2/3}}{27A^{1/3}} +\frac{2^{1/3}A^{1/3}}{27\times 37^{2/3}}}\right)z^2 \\&\quad+ \left(\sqrt{\frac{8}{27} +\frac{32\times 74^{2/3}}{27A^{1/3}} +\frac{2^{1/3}A^{1/3}}{27\times 37^{2/3}}}\right)^2 z^3\\&\quad+ \left(\sqrt{\frac{8}{27} +\frac{32\times 74^{2/3}}{27A^{1/3}} +\frac{2^{1/3}A^{1/3}}{27\times 37^{2/3}}}\right)^3 z^4\\&\quad+ \left(\sqrt{\frac{8}{27} +\frac{32\times 74^{2/3}}{27A^{1/3}} +\frac{2^{1/3}A^{1/3}}{27\times 37^{2/3}}}\right)^4 z^5+\ldots, \;\;z\in\mathbb{D};
\end{align*}
It can be easily shown that the function $ f_{1}\in \mathcal{S}^{*}_S $. A simple but tedious computations show that $ |H_{2,2}(F_{f_{1}}/2)| =|\gamma_{2}\gamma_{4} -\gamma^2_{3}|=(1272 + 113\sqrt{678})/ 32856 $, where
\begin{align*}
	\begin{cases}
		\gamma_{2}=\dfrac{1}{4}\left(\sqrt{\frac{8}{27} +\frac{32\times 74^{2/3}}{27A^{1/3}} +\frac{2^{1/3}A^{1/3}}{27\times 37^{2/3}}} \right)^2\vspace{1.5mm}\\
		\gamma_{3}=\dfrac{1}{6}\left(\sqrt{\frac{8}{27} +\frac{32\times 74^{2/3}}{27A^{1/3}} +\frac{2^{1/3}A^{1/3}}{27\times 37^{2/3}}}\right)^3\vspace{1.5mm}\\
		\gamma_{4}=\dfrac{1}{8}\left(\sqrt{\frac{8}{27} +\frac{32\times 74^{2/3}}{27A^{1/3}} +\frac{2^{1/3}A^{1/3}}{27\times 37^{2/3}}}\right)^4.
	\end{cases}
\end{align*}
Hence equality hods in \eqref{eq-2.6}. This completes the proof
\end{proof}
An example is presented below to demonstrate that the strict inequality in Theorem \ref{th-2.1} remains valid.
\begin{example}
Consider the function
\begin{align*}
	f_2(z)=\dfrac{z}{1 -z^2}=z +z^3 +z^5 +z^7 +\ldots
\end{align*}
It is easy to see that 
\begin{align*}
	\mbox{Re}\left(\dfrac{zf^{\prime}_{2}(z)}{f_{2}(z)-f_{2}(-z)}\right) =\dfrac{1}{2}\mbox{Re}\left(\dfrac{1+z^2}{1-z^2}\right)> 0.
\end{align*}
Hence, the function $ f_{2}\in\mathcal{S}^{*}_S $. We can easily compute and find that
\begin{align*}
	|H_{2,2}(F_{f_{2}}/2)|=\dfrac{1}{8}< \dfrac{(1272 +113\sqrt{678})}{32856}.
\end{align*}
\end{example}

\subsection{\bf Second Hankel determinant of logarithmic coefficients of $f\in \mathcal{K}_S$:}
We obtain the following sharp bound of $ H_{2,2}(F_{f}/2) $ for functions in the class $ \mathcal{K}_S $. 
\begin{thm}\label{th-2.3}
Let $ f\in \mathcal{K}_S $. Then
\begin{align*}
	|H_{2,2}(F_{f}/2)|\leq\dfrac{13}{1080}.
\end{align*}
The inequality is sharp.	
\end{thm}
\begin{proof}
Let $ f\in \mathcal{K}_S $ be of the form $ f(z)=z +\sum_{n=2}^{\infty} a_n z^n $, $ z\in \mathbb{D} $. Then by the definition of subordination there exits a Schwarz function $ w(z)=\sum_{n=1}^{\infty} c_n z^n $ such that
\begin{align}\label{eq-2.7}
	\dfrac{2(zf^{\prime}(z))^{\prime}}{(f(z)-f(-z))^{\prime}} =\dfrac{1 +w(z)}{1-w(z)}.
\end{align}
Comparing the coefficients on both side of \eqref{eq-2.7}, we obtain
\begin{align}\label{eq-2.8}
	\begin{cases}
		a_{2} =\dfrac{1}{2}c_1 ,\vspace{1.5mm}\\ a_{3}= \dfrac{1}{3}(c_2 + c^2_1),\vspace{1.5mm}\\ a_{4}=\dfrac{1}{8}(c_3 +3c_1 c_2 +2c^3_1),\vspace{1.5mm}\\ a_{5}=\dfrac{1}{10}(c_4 + 2c^2_2 +5c^2_1 c_2 +2c_1 c_3 +2c^4_1).
	\end{cases}
\end{align}
By utilizing \eqref{Leq-1.1} and \eqref{eq-2.8}, an easy calculation gives that
\begin{align}\label{eq-2.9}
	\nonumber H_{2,2}(F_{f}/2)&=\gamma_{2}\gamma_{4} -\gamma^2_{3} \\ &\nonumber=\dfrac{1}{288} \left(a^6_2 -36a^3_3 -72a^2_4 -6a^4_2 a_3 -12a^3_2 a_4 +72a_2 a_3 a_4 \right.\\&\nonumber\left. \quad+18a^2_2 a^2_3 -36a^2_2 a_5 +72a_3 a_5\right)\\&\nonumber =\dfrac{1}{276480} \left(175c^6_1 + 2508c^4_1 c_2 -180c^3_1 c_3 -432c_1 c_2 c_3 +5640c^2_1 c^2_2 \right.\\&\left. \quad+1440c^2_1 c_4 +3328 c^3_2 -1080c^2_3 +2304c_2 c_4\right).
\end{align}
Applying the triangle inequality to \eqref{eq-2.9} and using Lemma \ref{lem-2.1}, we obtain
\begin{align}\label{eq-2.10}
   \nonumber 276480|H_{2,2}(F_{f}/2)|&\leq 175|c_{1}|^6 +2508|c_{1}|^4|c_{2}|+ 180|c_{1}|^3\left(1-|c_1|^2 -\dfrac{|c_2|^2}{1+|c_1|}\right) \\&\nonumber \quad+432|c_{1}||c_{2}|\left(1-|c_1|^2 -\dfrac{|c_2|^2} {1+|c_1|}\right) +5640 |c_1|^2|c_2|^2 \\&\nonumber \quad+ 1440|c_1|^2 (1-|c_1|^2-|c_2|^2) +3328|c_2|^3 \\& \quad+2304|c_2|(1-|c_1|^2-|c_2|^2) +1080 \left(1-|c_1|^2 -\dfrac{|c_2|^2}{1+|c_1|}\right)^2.
\end{align}
Suppose that $ x=|c_{1}| $ and $ y=|c_{2}| $. Then it follows from \eqref{eq-2.10} that
\begin{align}\label{eq-2.11}
	276480|H_{2,2}(F_{f}/2)|\leq N(x,y),
\end{align}
where 
\begin{align*}
	N(x,y)&:=175x^6 +2508x^4y+ 180x^3 \left(1-x^2 -\dfrac{y^2}{1+x}\right)
	 +5640 x^2y^2\\& \quad+432xy\left(1-x^2 -\dfrac{y^2} {1+x}\right) + 1440x^2 (1-x^2-y^2) +3328y^3 \\& \quad+2304y(1-x^2-y^2) +1080 \left(1-x^2 -\dfrac{y^2}{1+x}\right)^2.
\end{align*}
Lemma \ref{lem-2.1} reveals that the variability range for a given pair $ (x,y) $ coincides with the set $\Omega=\{(x,y):0\leq x\leq 1, 0\leq y\leq 1-x^2\}$. Given this information, we proceed to find the maximum value of $ N(x,y) $ within the region $ \Omega $. Therefore, the critical point of $ N(x,y) $ satisfies the conditions
\begin{align*}
	\dfrac{\partial N}{\partial x}&= 1050x^5 -900x^4 -1440x^3 +540x^2 -1440x -1296x^2 y \\&\quad+10032x^3 y +8400xy^2 -4608xy +432y + \dfrac{108y^2(40x-5x^2 -4y)}{(1+x)}\\&\quad +\dfrac{36y^2(5x^3 +12xy -60x^2 +60)}{(1+x)^2} -\dfrac{2160y^4}{(1+x)^3}
\end{align*}
and
\begin{align*}
	\dfrac{\partial N}{\partial y}&=2508x^4 -432x^3 -2304x^2 +432x +2304 +3072y^2 + 8400x^2 y \\&\quad + \dfrac{72y(60x^2 -18xy -5x^3 -60)}{(1+x)} +\dfrac{4320y^3}{(1+x)^2}.
\end{align*}
By applying the analogous reasoning as in the proof of Theorem \ref{th-2.1}, we establish that the maximum of $N(x,y)$ is obtained on the boundary of $\Omega$. Specifically, on the boundary, it can be seen that
\begin{align*}
	&N(x,0)= 175x^6 -180x^5 -360x^4 +180x^3 -720x^2 +1080 \leq 1080 \;\;\mbox{for}\; 0\leq x\leq 1,\\& N(0,y)=1080y^4 +1024y^3 -2160y^2 +2304y +1080 \leq 3328 \;\;\mbox{for}\; 0\leq y\leq 1
\end{align*} 
and
\begin{align*}
	N(x,1-x^2)= 2175x^6 -4800x^4 -528x^2 +3328 \leq 3328 \;\;\mbox{for}\; 0\leq x\leq 1.
\end{align*}
Therefore, it is clear that
\begin{align*}
	\max_{(x,y)\in \Omega} N(x,y) =3328
\end{align*}
and from \eqref{eq-2.11}, we easily obtain
\begin{align}\label{eq-2.12}
	|H_{2,2}(F_{f}/2)|\leq \dfrac{13}{1080}.
\end{align}

In order to show the equality in \eqref{eq-2.12} sharp, we consider the function
\begin{align*}
	g_{1}(z)&=-\dfrac{1}{2\left(\frac{2}{5}\right)^{1/3}\left(\frac{13}{7}\right)^{1/6}}\log\left(1 -\left(2\left(\frac{2}{5}\right)^{1/3} \left(\frac{13}{7}\right)^{1/6}\right)z\right)\\&=z + \dfrac{1}{2}\left(2\left(\frac{2}{5}\right)^{1/3} \left(\frac{13}{7} \right)^{1/6}\right)z^2 + \dfrac{1}{3}\left(2\left(\frac{2}{5} \right)^{1/3} \left(\frac{13}{7}\right)^{1/6}\right)^2 z^3 \\&\quad+ \dfrac{1}{4}\left(2\left(\frac{2}{5}\right)^{1/3} \left(\frac{13}{7} \right)^{1/6}\right)^3 z^4 +\dfrac{1}{5}\left(2\left(\frac{2}{5} \right)^{1/3} \left(\frac{13}{7}\right)^{1/6}\right)^4 z^5 +\ldots, \;\;z\in\mathbb{D}.
\end{align*}
It can be shown that the function $ g_{1} $ belongs to the class $ \mathcal{K}_S $. Further, an easy computation shows that $ |H_{2,2}(F_{g_{1}}/2)| =|\gamma_{2}\gamma_{4}-\gamma^2_{3}|=13/1080 $, where
\begin{align*}
	\begin{cases}
		\gamma_{2}=\dfrac{5}{48}\left(2\left(\frac{2}{5}\right)^{1/3} \left(\frac{13}{7}\right)^{1/6}\right)^2\vspace{1.5mm}\\
		\gamma_{3}=\dfrac{1}{16}\left(2\left(\frac{2}{5}\right)^{1/3} \left(\frac{13}{7}\right)^{1/6}\right)^3\vspace{1.5mm}\\
		\gamma_{4}=\dfrac{251}{5760}\left(2\left(\frac{2}{5}\right)^{1/3} \left(\frac{13}{7}\right)^{1/6}\right)^4.
	\end{cases}
\end{align*}
This shows that the equality holds in \eqref{eq-2.12}. This completes the proof.
\end{proof}
An example of a function in the class $\mathcal{K}_S$ is provided to demonstrate Theorem \ref{th-2.3}. The example satisfies the conditions of the theorem and shows that the inequality is strict also.
\begin{example}
Consider the function
\begin{align*}
	g_{2}(z)=-\log(1-z)= z+ \frac{1}{2}z^2 +\frac{1}{3}z^3 +\frac{1}{4} z^4 +\frac{1}{5}z^5 +\ldots, \;\;z\in\mathbb{D}.
\end{align*}
It is easy to see that 
\begin{align*}
	\mbox{Re}\left(\dfrac{(zg^{\prime}_{2}(z))^{\prime}}{(g_{2}(z)-g_{2}(-z))^{\prime}}\right) =\dfrac{1}{2}\mbox{Re}\left(\dfrac{1+z}{1-z}\right)> 0.
\end{align*}
Therefore, the function $ g_{2}\in\mathcal{K}_S $. A simple computation shows that
\begin{align*}
	|H_{2,2}(F_{g_{2}}/2)|=\dfrac{35}{55296}<\dfrac{13}{1080}.
\end{align*}
\end{example}

\end{document}